\documentclass[11pt,a4paper,reqno]{amsart}
\usepackage[english]{babel}
\usepackage[applemac]{inputenc}
\usepackage[T1]{fontenc}
\usepackage{palatino}
\usepackage{cite}
\usepackage{float, verbatim}
\usepackage{amsmath}
\usepackage{amssymb}
\usepackage{amsthm}
\usepackage{amsfonts}
\usepackage{graphicx}
\usepackage{mathtools}
\usepackage{enumitem}
\usepackage[colorlinks = true, citecolor = black]{hyperref}
\usepackage{color}
\usepackage{tikz}
\usepackage{xcolor}
\usepackage{pgfplots}
\usepackage{caption}
\usepackage{subcaption}
\usepackage{enumerate}
\usepackage{autobreak}

\newcommand{\Assouad}{\dim_{\mathrm{A}}}
\newcommand{\ubox}{\overline{\dim_{\mathrm{B}}}}
\newcommand{\lbox}{\underline{\dim_{\mathrm{B}}}}
\newcommand{\nbox}{\dim_{\mathrm{B}}}

\newcommand{\Haus}{\dim_{\mathrm{H}}}

\newtheorem*{thm*}{Theorem}

\newtheorem{thm}{Theorem}[section]

\newtheorem{lma}[thm]{Lemma}

\newtheorem{defn}[thm]{Definition}

\newtheorem{conj}[thm]{Conjecture}
\newtheorem{rem}[thm]{Remark}
\newtheorem{ques}[thm]{Question}

\begin{document}

\title{Weak tangents and level sets of Takagi functions}

\author{Han Yu}
\address{Han Yu\\
School of Mathematics \& Statistics\\University of St Andrews\\ St Andrews\\ KY16 9SS\\ UK \\ }
\curraddr{}
\email{hy25@st-andrews.ac.uk}
\thanks{}



\date{}

\dedicatory{}

\begin{abstract}
In this paper we study some properties of Takagi functions and their level sets. We show that for Takagi functions $T_{a,b}$ with parameters $a,b$ such that $ab$ is a root of a Littlewood polynomial, there exist large level sets. As a consequence we show that for some parameters $a,b$, the Assouad dimension of graphs of $T_{a,b}$ is strictly larger than their upper box dimension. In particular we can find weak tangents of those graphs with large Hausdorff dimension, larger than the upper box dimension of the graphs.
\end{abstract}

\maketitle
\allowdisplaybreaks

\section{Introduction}
In this paper we study graphs of the following functions,
\[
T_{a,b}(x)=\sum_{n=0}^{\infty} a^nT(b^n x),
\]
where $a,b$ are real parameters $a,b$ such that $a<1,b>1, ab\geq 1$ and $T:\mathbb{R}\to\mathbb{R}$ is the tent map which has period $1$ and defined on the unit interval as follows,
\[
T(x)=
\begin{cases}
x & x\in [0,\frac{1}{2}] \\
1-x & x\in [\frac{1}{2},1].
\end{cases}
\]
Such functions $T_{a,b}$ are called the Takagi functions. Originally the Takagi function was referred to $T_{1/2,2}$ but no confusion should appear if we also call $T_{a,b}$ the Takagi functions. There has been a lot of interest in the Hausdorff and box counting dimensions of the graphs of such functions. For the box dimensions we know from \cite[Section 2]{KPY} and \cite[Theorem 2.4]{Ba} that the upper box dimension of graphs of these functions $T_{a,b}$ can be computed by the following formula
\[
B=2+\frac{\ln a}{\ln b}=1+\frac{\ln ab}{\ln b}.
\]
The Hausdorff dimensions of graphs of these functions are harder to obtain, see \cite{WXS},\cite{BBR} and the references therein for more recent results on related questions.

One of the results of this paper is about the Assouad dimension of some Takagi functions. In what follows,	for a function $f:\mathbb{R}\to\mathbb{R}$, we denote the following set
\[
\Gamma_f=\{(x,y)\in\mathbb{R}^2:x\in [0,1], y=f(x)\},
\]
to be the graph of $f$ over the interval $[0,1]$.
\begin{thm}[Assouad dimension]\label{MAIN}
	Let the product $ab>1$ be a root of a Littlewood polynomial of degree $k-1$, namely
	\[
	\sum_{n=0}^{k-1}\epsilon_n(ab)^n=0,
	\]
	for a sequence $\{\epsilon_n\}_{n\in \{0,\dots,k-1\}}$ over $\{0,1\}$.
	Furthermore, if $b$ is an integer greater than $2$, then we have the following result,
	\[
	\dim_{\mathrm{A}} \Gamma_{T_{a,b}}\geq 1+\frac{1}{k}.
	\]
\end{thm}
Notice that by keeping the product $ab$ unchanged and making $b$ larger, this lower bound can be larger than the upper box dimension $\frac{\ln ab}{\ln b}$ for large $b$. For example, when we choose parameters such that $ab=\frac{\sqrt{5}+1}{2}, b=8$, then $\ubox\Gamma_{T_{a,b}}\approx 1.23$ and $\Assouad\Gamma_{T_{a,b}}\geq 4/3.$ 

One consequence of Theorem \ref{MAIN} is that there exist large weak tangents of the graphs of Takagi functions. See Section \ref{Pre} for more details about the notions of dimensions, definition of weak tangent and some basic properties.
\begin{lma}[Weak tangent]\label{Weaktangent}
	Let $a,b$ be as in the statement of Theorem \ref{MAIN}, then there exists weak tangent $E$ of $\Gamma_{T_{a,b}}$ such that
	\[
	\Haus E=\Assouad \Gamma_{T_{a,b}}\geq\ubox\Gamma_{T_{a,b}}.
	\]
	The last inequality can be strict.
\end{lma}

Theorem \ref{MAIN} follows from the existence of large level set of graphs $\Gamma_{T_{a,b}}$ and we think this result is interesting on its own.

\begin{thm}\label{LEVEL}
	Let $T_{a,b}$ be as in the statement of Theorem \ref{MAIN} but we also allow $ab=1$. For each $y\in\mathbb{R}$ we define the following level set
	\[
	L(y)=\{x\in [0,1]: T_{a,b}(x)=y\}\times \{y\}.
	\]
	There exists $y\in\mathbb{R}$ such that
	\[
	\Haus L(y)\geq \frac{1}{k}.
	\]
\end{thm}
 
The restriction of the product $ab$ as a certain algebraic integer seems to be really strong, however, with some effort we can show that those algebraic integers are dense in $[\frac{1}{2},2]$.
\begin{thm*}
	Let $L$ be the set of algebraic integers which are roots of Littlewood polynomials namely, $x\in\mathbb{C}$ and there exist a finite sequence $\epsilon_n\in\{\pm 1\}$ such that
	\[
	\sum_{n=0}^{k-1}\epsilon_nx^n=0.
	\]
	Then $L\cap [\frac{1}{2},2]$ is dense in $[\frac{1}{2},2]$.
\end{thm*}
Proofs of the above result and its generalizations can be found in \cite{OP}, \cite{Bandt} and \cite{BY}.
	
\section{Discussions and future work}
In this section we give some backgrounds of Theorem \ref{MAIN} and \ref{LEVEL}. We also pose some questions which are related with the results in this paper.
\subsection{Assouad dimensions of graphs of functions}
Theorem \ref{MAIN} deals with the Assouad dimension of some Takagi functions. It is natural to think about the Assouad dimension of other nowhere differentiable functions, for example  Weierstrass functions and graphs of the Wiener process. For the latter, we have the following result (\cite[Theorem 2.2]{HY}).
\begin{thm}[HY17]
	The graph of the Wiener process $W(.)$ over the unit interval has the Assouad dimension equal to $2$ almost surely.
\end{thm}
We have not completely determined the Assouad dimension of any Takagi function yet. We only showed that the Assouad dimension can be strictly larger than the upper box dimension for graphs of Takagi functions. Based on the above theorem we think that the Assouad dimension of all Takagi functions should be $2$.
\begin{conj}
	For $a,b\in\mathbb{R}^+$ and $ab>1$, we have that for all Takagi functions $T_{a,b}$
	\[
	\Assouad \Gamma_{T_{a,b}}=2.
	\]
\end{conj}
\subsection{Level sets of Takagi functions}
 For more details about the level sets of Takagi functions, see \cite{JL} and \cite{AK}. Notice that if we set $a=0.5, b=2$ then we can find level set of $T_{a,b}$ with Hausdorff dimension at least $0.5$. This is sharp, see \cite{AMO}. For other values of $a,b$ for example $a=(\sqrt{5}+1)/16, b=8$ we see that we can find a level set with Hausdorff dimension at least $1/3$ and we do not know whether this is sharp.
 \begin{ques}
 	What is the largest level set of the Takagi function
 	$
 	T_{\frac{\sqrt{5}+1}{16},8}
 	$
 	in terms of the Hausdorff dimension?
 \end{ques}
\section{Notation}\label{Notation}
\begin{itemize}
	\item[1.] For a real number $x\in\mathbb{R}$ we use the symbol $x^+$ to denote a number $x+\epsilon$ where $\epsilon>0$ is some fixed positive number whose value can be chosen freely and we will point out the specific value of $\epsilon$ when necessary. Similarly, we use $x^-$ for a number smaller but close to $x$.
	\item[2.]For a function $f:\mathbb{R}\to\mathbb{R}$, the following set
	\[
	\Gamma_f=\{(x,y)\in\mathbb{R}^2:x\in [0,1], y=f(x)\},
	\]
	is called the graph of $f$ over the interval $[0,1]$.
	\item[3.] For a real number $x$, we use $\lfloor x\rfloor$ to denote the greatest integer that is not strictly larger than $x$.
\end{itemize}

\section{Preliminaries}\label{Pre}

We will now introduce some notions of dimensions which will be used in this paper. We use $N_r(F)$ for the minimal covering number of a bounded set $F$ in $\mathbb{R}^n$ with balls of side length $r>0$. 
\subsection{Hausdorff dimension}

For any $s>0$ and $\delta>0$ define the following quantity:
\[
\mathcal{H}^s_\delta(F)=\inf\left\{\sum_{i=1}^{\infty}(\mathrm{diam} (U_i))^s: \bigcup_i U_i\supset F, \mathrm{diam}(U_i)<\delta\right\}.
\]
The $s$-Hausdorff measure of $F$ is
\[
\mathcal{H}^s(F)=\lim_{\delta\to 0} \mathcal{H}^s_{\delta}(F),
\]
and Hausdorff dimension of $F$ is
\[
\Haus F=\inf\{s\geq 0:\mathcal{H}^s(F)=0\}=\sup\{s\geq 0: \mathcal{H}^s(F)=\infty          \}.
\]
\subsection{upper box dimension}
The upper box dimension of $F$ is
\[
\overline{\nbox}=\limsup_{r\to 0}\left(-\frac{\log N_r(F)}{\log r}\right).
\]
\subsection{Assouad dimension and weak tangents}
The \textit{Assouad dimension} of $F$ is 
\begin{align*}
\Assouad F = \inf \Bigg\{ s \ge 0 \, \, \colon \, (\exists \, C >0)\, (\forall & R>0)\,  (\forall r \in (0,R))\, (\forall x \in F) \\ 
&N_r(B(x,R) \cap F) \le C \left( \frac{R}{r}\right)^s \Bigg\}
\end{align*}
where $B(x,R)$ denotes the closed ball of centre $x$ and radius $R$.

An important tool for studying the Assouad dimension is \emph{weak tangents} introduced in \cite{MT} and \emph{microsets} in \cite{Fu}. The next definition appeared in \cite[Definition 1.1]{Fr}.

\begin{defn}
	Let $X\in\mathcal{K}(\mathbb{R}^n)$ be a fixed reference set (usually the closed unit ball or cube) and let $E,F\subset\mathbb{R}^n$ be compact sets. Suppose there exists a sequence of similarity maps $T_k:\mathbb{R}^n\to\mathbb{R}^n$ such that $d_\mathcal{H}(E,T_k(F)\cap X)\to 0$ as $k\to 0$. Then $E$ is called a \emph{weak tangent} of $F$.
\end{defn}

Here $(\mathcal{K}(\mathbb{R}^n),d_{\mathcal{H}})$ is a complete metric space with the Hausdorff metric, namely, for two compact subsets $A,B\subset\mathbb{R}^n$ is defined by
\[
d_{\mathcal{H}}(A,B)=\inf\{\delta>0: A\subset B_{\delta}, B\subset A_{\delta}\},
\]
where for any compact set $C\subset\mathbb{R}^n$
\[
C_{\delta}=\{x\in\mathbb{R}^n: |x-y|<\delta \text{ for some } y\in C\}.
\]

For Lemma \ref{Weaktangent} we need  the following  result \cite[Proposition 5.7]{KOR}

\begin{thm}[KOR]\label{ThKOR}
	Let $F$ be a compact set with $\Assouad F=s$. Then there exist a weak tangent $E$ of $F$ such that
	\[
	\Haus E=s.
	\]
	In other words, we have
	\[
	\Assouad F=\max\{\Haus E: E \text{ is a weak tangent of } F\}.
	\]
\end{thm}

\subsection{Assouad spectrum}
\begin{defn}[Fraser and Yu \cite{JMF3}]\label{ASP}
	\begin{eqnarray*}
		\Assouad^{\theta} F &=& \inf \bigg\{ \alpha \  : \   (\exists C>0) \, (\exists \rho>0) \, (\forall 0<R\leq \rho) \,  (\forall x \in F) \\ \\
		&\,& \qquad \qquad  \qquad \qquad   N_{R^{1/\theta}} \big( B(x,R) \cap F \big) \ \leq \ C \left(\frac{R}{R^{1/\theta}}\right)^\alpha \bigg\}.
	\end{eqnarray*}
\end{defn}

As $\theta$ ranges in $(0,1)$ the function $\Assouad^{\theta} F$ with respect to $\theta$ is called the \emph{Assouad spectrum} of $F$. For more background introduction of Assouad dimension/spectrum and how they are related to homogeneity of fractal sets see \cite{JMF}, \cite{JMF2} as well as \cite{JMF3}. It is known (\cite[Theorem 1.3]{JMF3}) that for any $\theta\in (0,1)$  we have
\[
\ubox F\leq\Assouad^{\theta} F\leq\Assouad F.\label{Di}\tag{1}
\]

\subsection{Covering by disjoint cubes}

For convenience, in this paper we will count covering number with disjoint squares rather than balls.  We denote $S(a,R)$ for $a\in\mathbb{R}^2, R>0$ as the square centred at $a$ with side length $2R$ whose sides are parallel to the coordinate axis. Since we are dealing with graph of functions, the choice of axis is natural. We denote the following covering number,
\begin{eqnarray*}
& &N(F\cap S(a,R),r)=\bigg|\bigg\{(i,j)\in\mathbb{Z}^2\cap [0,\lfloor R/r \rfloor+1]^2 :\\
& &S((a-R/2+r/2+ir,a-R/2+r/2+jr),r) \cap F\neq\emptyset\bigg\}\bigg|.
\end{eqnarray*}

This is equivalent to $N_r(F\cap B(a,R))$ in the sense that there exists a constant $C>0$ such that for all $a\in F, 0<r<R<1$ we have the following inequality, 
\[
C^{-1}N_r(F\cap B(a,R)) \leq N(F\cap S(a,R),r)\leq C N_r(F\cap B(a,R)).
\]

\subsection{Some properties of Takagi functions}
In this paper we will use the following result whose proof can be found in \cite{HL} and we use the version presented in \cite[Theorem 2.4]{Ba}.

\begin{lma}\label{LL2}
Let $T:\mathbb{R}\to\mathbb{R}$ be a continuous piecewise $C^1$ and periodic function. Then the following function
\[
T_{a,b}(x)=\sum_{n=0}^{\infty} a^nT(b^n x)
\]
must satisfy one of the two properties, 
\item{1}:
$T_{a,b}$ is piecewise $C^1$.
\item{2}:
For a positive constant $C>0$ and any interval $J\subset\mathbb{R}$ we have the following inequality,
\[
\sup_{x,y\in J} |T_{a,b}(x)-T_{a,b}(y)|\geq C|J|^{-\frac{\ln a}{\ln b}}.
\]
Notice that if $a<1,ab>1$ then $-\frac{\ln a}{\ln b}\in (0,1)$ and we see that if $|J|<1$ then
\[
\sup_{x,y\in J} |T_{a,b}(x)-T_{a,b}(y)|\geq C|J|.
\]
\end{lma}
\begin{rem}
When $T$ is the tent map defined in the beginning of the first section, it is known that when $a<1,ab\geq 1$, the function $T_{a,b}$ is nowhere differentiable therefore only the second property of lemma \ref{LL2} can be true.
\end{rem}

\section{Large level sets, proof of Theorem \ref{LEVEL}}
We show that there exist large level sets for function $T_{a,b}$ with certain parameters $a,b$. Since $ab$ is a root of a Littlewood polynomial we see that
\[
\sum_{i=0}^k {\epsilon_i} (ab)^i=0
\]
for an integer $k\geq 1$ and some choice of $\epsilon_i\in \{\pm 1\}$.
Next we consider the first $k$ terms partial sum
\[
F_1(x)=\sum_{n=0}^{k-1}a^nT(b^n x).
\]
The derivative of the above function is not continuous at $x=m b^{-k}$ for integers $m$. Let us now assume that $x$ is an irrational number then the derivative is 
\[
F'_1(x)=\sum_{i=0}^{k-1} \epsilon_i(x) (ab)^i,
\]
where $\epsilon_i(x)\in \{\pm 1\}$ depends on the $b$-nary expansion of $x$. In particular, if
\[
x=0.b_1b_2\dots
\]
then 
\[
\epsilon_i(x)=
\begin{cases}
1 & b_i\in [0,b/2] \\
-1 & b_i\in (b/2,1]
\end{cases}
\]
Therefore we can find at least $2$ disjoint intervals of length $\frac{1}{2b^{k-1}}$ where $F_1'(x)=0$ and $F_1(x)=a_1$ is for a constant $a_1\geq 0$ on those two intervals. Indeed, when
\[
\sum_{i=0}^{k-1} {\epsilon_i} (ab)^i=0,
\]
we also have
\[
-\sum_{i=0}^{k-1} {\epsilon_i} (ab)^i=0.
\]
So there are at least two intervals we can find and the union is symmetric with respect to the line $\{x=0.5\}$. Then because $F_1$ is also symmetric with respect to the line $\{x=0.5\}$ we see that $F_1(.)$ takes the same value on those two intervals, say, $I_1$ and $I_2$.

We consider the next $k$ terms sum
\[
F_2(x)=\sum_{n=k}^{2k-1}a^nT(b^n x).
\]
Then we can find $b$ many intervals of length $1/2b^{2k-1}$ in $I_1,I_2$ such that the above sum stays constant $a_2\geq 0$ on those intervals. To see this, consider $I_1$, which is an interval of length $1/2b^k$. Now observe the following
\[
F_2(x)=\sum_{n=k}^{2k-1}a^nT(b^n x)=\sum_{n=0}^{k}a^{n+k}T(b^{k}b^n x)=a^{k} F_1(b^k x).
\]
Therefore the graph of $F_2$ is an affine copy, or intuitively speaking, a narrowed version of the graph of $F_1$. Then we see that there are exactly $b$ many intervals in $I_1$ of length $1/2b^{2k-1}$ such that $F_2$ equals to $a_2$ on all those intervals. Indeed, over any interval the form $[l/b^{k-1},(l+1)/b^{k-1}]$ the graph of $F_2$ has $2b$ many platforms of the same level. That is to say, we can find $2b$ many $1/2b^{2k-1}$ length intervals on which $F_2(x)=a_2$. Since $I_1$ is only a half $1/b^{k-1}$ length interval, therefore we can find $b$ many platforms over $I_1$. Here we used the mirror symmetry of $F_2$.

We can apply the above argument to $j$-th $k$-terms partial sums for each $j\geq 2$ and as a result we can find a Cantor set $C$ such that $T_{a,b}(C)=\{c\}$ for a constant $c$. This Cantor set is a self-similar set satisfying open set condition and its Hausdorff dimension is $1/k$ (the contraction ratio is $1/b^k$ and branching number is $b$, see for example \cite[Theorem 9.3]{Fa}). Thus we have proved theorem \ref{LEVEL}.
\section{Squashing and counting, proof of Theorem \ref{MAIN}}\label{SAC}
In order to deal with the Assouad dimension of graphs of Takagi functions we need to handle the following quantity

\[
N(S(x,R/2)\cap\Gamma_{T_{a,b}},r).
\]

The situation is not too bad when we want to deal with the above quantity for $\Gamma_{f+g}$ with one of the functions, say $f$, is Lipschitz continuous. The following result is a localized and quantitative version of \cite[Lemma 2.1,2.2]{FF}.
\begin{lma}\label{ADD}
If $f:[0,1]\to\mathbb{R}$ is Lipschtz continuous with Lipschitz constant $M>0$ and $g:[0,1]\to\mathbb{R}$ is continuous then we have the following inequality for $0<r<R<1$ whenever $\frac{R}{r}$ is an integer,
\begin{eqnarray*}
& &\sup_{a\in [0,1]\times\mathbb{R}}N(S(a,R/2)\cap\Gamma_{f+g},r)\geq \\ & &\frac{1}{M+2}\sup_{a\in [0,1]\times\mathbb{R}}N(S(a,R/2)\cap\Gamma_{g},r)-\frac{M+2}{\lfloor M\rfloor+2}\frac{R}{r}.
\end{eqnarray*}
\end{lma}
\begin{rem}\label{ADDr}
For the case when $R/r$ is not an integer, we can replace $r$ with a larger value
\[
r'=R \frac{1}{\lfloor R/r \rfloor}.
\]
As we will eventually choose $R/r$ to be arbitrarily large, $r'$ and $r$ are essentially the same. For any $\delta>0$, if $R/r$ is large enough the following relation holds
\[
(1-\delta)r<r'<(1+\delta)r.
\]
Then the inequality of this theorem holds with $1/(M+2)$ being replaced by some other constant which depends only on $M$.
\end{rem}

\begin{proof}[Proof of lemma \ref{ADD}]
Let $a\in\mathbb{R}^2$ and consider the square $S(a,R/2)$. For any $r<R$ we consider the following rectangles for $i=0,1,\dots,\frac{R}{r}-1$
\[
S_i=[a-R/2+ir,a-R/2+(i+1)r]\times [a-R/2,a+R/2]\subset S(a,R/2).
\]
Each rectangle $S_i$ contains the following squares for $j=0,1,\dots,\frac{R}{r}-1$
\[
S_{ij}=[a-R/2+ir,a-R/2+(i+1)r]\times [a-R/2+jr,a-R/2+(j+1)r].
\]
Now if $S_{ij}\cap\Gamma_{g}\neq\emptyset$ we colour it black, otherwise we colour it white. Let $n_i\geq 0$ denote the number of black squares among $S_{ij},j=0,1,2,\dots,\frac{R}{r}-1$. By continuity of $g$, the fact that we have $n_i$ black squares implies the following inequality
\[
\sup_{x,y\in [a-R/2+ir,a-R/2+(i+1)r]} |g(x)-g(y)|\geq (n_i-2)r.
\]
By Lipschitz property of function $f$ we see that
\[
|f(x)-f(y)|\leq M|x-y|,
\]
this implies that
\[
\sup_{x,y\in [a-R/2+ir,a-R/2+(i+1)r]} |f(x)-f(y)|\leq Mr.
\]
Then we see that
\[
\sup_{x,y\in  [a-R/2+ir,a-R/2+(i+1)r]} |f(x)+g(x)-f(y)-g(y)|\geq (n_i-2-M)r.
\]
So we see that to cover the set
\[
\{(x,y+f(x))\in\mathbb{R}^2 : (x,y)\in S_i\cap\Gamma_{g}\}
\]
 we need at least
\[
n_i-2-M
\]
many squares of side length $r$. Summing over all $i$ we see that to cover the set
\[
\{(x,y+f(x))\in\mathbb{R}^2 : (x,y)\in S(a,R/2)\cap\Gamma_{g}\}
\]
we need at least
\[
\sum_{i} n_i-(M+2)\frac{R}{r}
\]
many squares with side length $r$.

The next fact to notice is that the following set
\[
\{(x,y+f(x))\in\mathbb{R}^2 : (x,y)\in S(a,R/2)\cap\Gamma_g\}
\]
is contained in a $R\times (M+1)R$ rectangle. This rectangle can be covered by $\lfloor M\rfloor+2$ squares with side length $R$, so for at least one of the $\lfloor M\rfloor+2$ squares need at least
\[
\frac{\sum_{i} n_i-(M+2)\frac{R}{r}}{\lfloor M\rfloor+2}
\]
many squares with side length $r$ to cover.

It is then easy to see that
\[
\{(x,y+f(x))\in\mathbb{R}^2 : (x,y)\in S(a,R/2)\cap\Gamma_{g}\}\subset \Gamma_{f+g}.
\]
This implies that
\[
\sup_{a'\in [0,1]\times\mathbb{R}}N(S(a',R/2)\cap\Gamma_{f+g},r)\geq
\frac{\sum_{i} n_i-(M+2)\frac{R}{r}}{\lfloor M\rfloor+2},
\]
and since  $\frac{R}{r}$ is integer we see that
\[
\sum_{i} n_i=N(S(a,R/2)\cap\Gamma_{g},r).
\]
It follows that
\[
\sup_{a'\in [0,1]\times\mathbb{R}}N(S(a',R/2)\cap\Gamma_{f+g},r)\geq \frac{1}{M+2}N(S(a,R/2)\cap\Gamma_{g},r)-\frac{M+2}{\lfloor M\rfloor+2}\frac{R}{r},
\]
We can take the supreme of $a$ on the right hand side of the above inequality and the lemma concludes.
\end{proof}

Now we can move on dealing with the Assouad dimension of Takagi functions. We shall need the following lemma.

\begin{lma}\label{POWERBOUND}
	For any number $1<\beta<2$ and all integer $k$, there exists a sequence $\epsilon_n$ of $\pm 1$ such that $|\sum_{n=0}^{k}\epsilon_n \beta^n| \leq \frac{1}{\beta-1}$.
\end{lma}

\begin{proof}
	Let $k>0$ be an integer, consider the following two functions forming an IFS (known as iterated function system, see for example \cite[chapter 9]{Fa}),
	
	\[
	f_{-1}(x)=\beta x-1,
	f_{1}(x)=\beta x+1.
	\]
	Then for any sequence $\epsilon_n\in\{\pm 1\}$ with $n=0,1,2,\dots,k-1$ we can define an iteration by
	\[
	f_{\epsilon_{0}}\circ\dots\circ f_{\epsilon_{k-2}}\circ f_{\epsilon_{k-1}},
	\]
	and notice that	
	\[
	f_{\epsilon_{0}}\circ\dots\circ f_{\epsilon_{k-2}}\circ f_{\epsilon_{k-1}}(0)=\sum_{n=0}^{k-1}\epsilon_n \beta^n.
	\]
	So as long as we can find an iteration of this IFS such that the trajectory of $0$ stays bounded by $\frac{1}{\beta-1}$ the existence of a sequence $\epsilon_n$ will follow. Now since $f_1(0)=1<\frac{1}{\beta-1}$ which is the intersection of the line $y=f_{1}(x)$ and $y=x$, we can apply function $f_{1}$ before the value exceeds $\frac{1}{\beta-1}$ and apply $f_{-1}$ before the value drops below $-\frac{1}{\beta-1}$. More precisely, we put $x_0=0$ and if we find $x_i\in (-1/(\beta-1),1/(\beta-1))$ then if $f_1(x_i)<1/(\beta-1)$ we set $x_{i+1}=f_1(x_i)$ otherwise we set $x_{i+1}=f_{-1}(x_i)$. We need to check $x_{i+1}\in (-1/(\beta-1),1/(\beta-1))$ as well. In fact, if $x_{i+1}>1/(\beta-1)$ then $f_{-1}(x_i)>1/(\beta-1)$ and this implies that $x_i>1/(\beta-1)$. If $x_i<-1/(\beta-1)$ then either $f_1(x_i)<-1/(\beta-1)$ or $f_{-1}(x_i)<-1/(\beta-1)$ in the first case we have $x_i<-1/(\beta-1)$. The later case implies that $x_i<(2-\beta)/\beta(\beta-1)$ but then $f_1(x_i)<1/(\beta-1)$. So in any case $x_{i+1}\in (-1/(\beta-1),1/(\beta-1))$ as required. This procedure gives us an sequence $\epsilon_n$, with $n=0,1,2,\dots,k-1$ for any integer fixed $k-1$ such that
	
	\[
	\left|\sum_{n=0}^{k-1}\epsilon_n \beta^n\right| \leq \frac{1}{\beta-1}.
	\]
\end{proof}

Notice that when $\beta$ is a root of a Littlewood polynomial it is necessary that $0.5\leq |\beta|\leq 2$. Therefore if parameters of $T_{a,b}$ are as stated in Theorem \ref{MAIN} then $1\leq ab\leq 2$ and therefore the result of Lemma \ref{POWERBOUND} holds for $ab$. Now we have all the ingredients needed to prove Theorem \ref{MAIN} however we find it convenient to introduce the following general result.

\begin{lma}\label{SC}[Squash and count]
Suppose $T_{a,b}(x):\mathbb{R}\to\mathbb{R}$  is a function of the following form
\[
T_{a,b}(x)=\sum_{n=0}^{\infty}a^nT(b^n x),
\]
where $T(x):\mathbb{R}\to\mathbb{R}$ is a piecewise $C^1$ continuous function with period $1$ and $a>0,b>0,ab>1$. Suppose the following two conditions holds:
\item[1, \emph{(interval with slow changing)}]: There exists a positive constant $C_1>0$ such that  for any integer $M>0$, there is a integer $k$ such that $J_k=(\frac{k}{b^{M+1}},\frac{k+1}{b^{M+1}})$ and for all $x_1,x_2\in J_k$ the following condition holds,
\[
\bigg|\sum_{n=0}^{M}a^nT(b^n x_1))-\sum_{n=0}^{M}a^nT(b^n x_2))\bigg|<C_1|x_1-x_2|.
\]
\item[2, \emph{(large level set)}]:
There exists a level set $L\subset [0,1]$ with lower box dimension at least $D$, namely,
\[
\exists y\in\mathbb{R}, \lbox L(y)=\lbox\{x\in [0,1]: T_{a,b}(x)=y\}\geq D.
\]
Then we have the following result,
\[
\Assouad \Gamma_{T_{a,b}}\geq D+1.
\]
\end{lma}
\begin{proof}
For any positive integer $M>0$, we can find an integer $k$ and $x_0=\frac{k}{b^{M+1}}$ such that $(x_0,x_0+\frac{1}{b^{M+1}})\subset [0,1]$ and on this subset we have the following condition for the oscillation of the following $M$-th partial sum
\[
\bigg|\sum_{n=0}^{M}a^nT(b^n x_1))-\sum_{n=0}^{M}a^nT(b^n x_2))\bigg|<C_1|x_1-x_2|,
\]
where $x_1,x_2\in (x_0,x_0+\frac{1}{b^{M+1}})$. Then we can write that

\begin{eqnarray*}
T_{a,b}(x)=\sum_{n=0}^{\infty}a^nT(b^n x)&=&\sum_{n=0}^{M}a^nT(b^n x)+\sum_{n=M+1}^{\infty}a^nT(b^n x)\\
&=& F_M(x)+G_M(x),
\end{eqnarray*}
where the functions $F_M,G_M$ are the first sum and second sum in the third expression. Then we see that the graph $\Gamma_{G_M}$ is actually a 'squashed' version of $\Gamma_{T_{a,b}}$, namely we have the following relation
\[
X_M(\Gamma_{T_{a,b}})=\Gamma_{G_M},
\]
where the linear transformation $X_M:\mathbb{R}^2\to \mathbb{R}^2$ is defined to be as follows
\[
X_M(x,y)=\left(\frac{x}{b^{M+1}},a^{M+1}y,\right).
\]

We see that since $ab>1$ this linear transformation squashes a square to a very thin rectangle for large enough $M$. Now we concentrate on the strip
\[
S=\left(x_0,x_0+\frac{1}{b^{M+1}}\right)\times\mathbb{R}.
\]
The graph $\Gamma_{G_M}$ over this strip is the squashed version of the graph $\Gamma_{T_{a,b}}$ over $[0,1]$. We want to find a $\frac{1}{b^{M+1}}$-square contained in this strip such that we need a reasonably large amount of $r$-squares to cover $\Gamma_{G_M}$, where $r>0$ is number that will be specified later. Now because of the bijective linear map $X_M$, covering a $\frac{1}{b^{M+1}}$-square with $r$-square in $\Gamma_{G_M}$ is the same thing as covering the original graph $\Gamma_{T_{a,b}}$ over $[0,1]$ inside a $1\times\frac{1}{(ab)^{M+1}}$ -rectangle with $rb^{M+1}\times\frac{r}{a^{m+1}}$-rectangles.

Now consider the level set $L$ with lower box dimension $D$ mentioned in the second condition. For some real number $y'$ we have that

\[
L=L(y')=\{x\in [0,1]: T_{a,b}(x)=y'\}\times\{y'\},
\]
then we do box counting in the  $1\times\frac{1}{(ab)^{M+1}}$-rectangle containing $L$. For any $(x,y')\in L$ the graph $\Gamma_{T_{a,b}}$ intersects the middle axis of the rectangle at $(x,y')$, so by lemma \ref{LL2} we see that there exists a positive constant $C_3>0$ such that the projection of graph $\Gamma_{T_{a,b}}$ inside any $\frac{1}{(ab)^{M+1}}$-square centred in $L$ to the vertical axis has length at least
\[
\min\bigg(C_3\frac{1}{(ab)^{M+1}},\frac{1}{(ab)^{M+1}}\bigg).
\]
If $M$ is large enough we need at least $\left(\frac{1}{ab}\right)^{-D^-(M+1)}$ many $\frac{1}{(ab)^{M+1}}$-square to cover this  $1\times\frac{1}{(ab)^{M+1}}$-rectangle because for covering the level set $L$ we already need that many squares. Since each square is sufficiently occupied by $\Gamma_{T_{a,b}}$ in the sense that the curve occupies at least $\min(C_3,1)$ portion of the vertical length. This means that for each such square we need some constant times
\[
\frac{1/(ab)^{M+1}}{r/a^{M+1}}=\frac{1}{rb^{M+1}}
\]
many  $rb^{M+1}\times\frac{r}{a^{m+1}}$-rectangle to cover. Now we choose the following value for $r$
 \[
r=\frac{1}{(ab^2)^{M+1}}=\left(\frac{1}{(ab)^{M+1}}\right)^{\frac{\ln ab^2}{\ln ab}}=\left(\frac{1}{(ab)^{M+1}}\right)^{\frac{B}{B-1}}=\left(\frac{1}{b^{M+1}}\right)^B,
\]
where $B=2+\frac{\ln a}{\ln b}$ is the upper box dimension of $\Gamma_{T_{a,b}}$. We denote $R=\frac{1}{b^{M+1}}$ and note that $r=R^{\frac{1}{\theta}}$ with $\theta=\frac{1}{B}$ and we get the following relation
\[
\sup_{a\in\mathbb{R}^2}N(S(a,R/2)\cap\Gamma_{G_M},R^{\frac{1}{\theta}})\geq \left(\frac{1}{ab}\right)^{-D^-(M+1)}\frac{1}{rb^{M+1}}=\left(\frac{R}{r}\right)^{1+D^-}.
\]
The above inequality holds for arbitrarily large $M$ and therefore it holds also for arbitrarily small $R,R^{\frac{1}{\theta}}$. This is a covering property for $\Gamma_{G_M}$ and we can translate it to a covering property for $\Gamma_{T_{a,b}}$. By using Lemma \ref{ADD} and Remark \ref{ADDr} together with condition $(1)$ we see that there exist a constant $C>0$ such that
\begin{eqnarray*}
& &\sup_{a\in [x_0,x_0+\frac{1}{b^{M+1}}]\times\mathbb{R}}N(S(a,R/2)\cap\Gamma_{T_{a,b}},r)\\
&\geq& C\sup_{a\in [x_0,x_0+\frac{1}{b^{M+1}}]\times\mathbb{R}}N(S(a,R/2)\cap\Gamma_{G_M})-C\frac{R}{r}\\
&\geq& C\left(\frac{R}{r}\right)^{1+D^-}-C\frac{R}{r}.
\end{eqnarray*}
Then by definition of the Assouad spectrum and the inequality (\ref{Di}) in the first section we see that
\[
\Assouad \Gamma_{T_{a,b}}\geq \Assouad^\frac{1}{B} \Gamma_{T_{a,b}}\geq 1+D.
\]
This concludes the proof.

\end{proof}

We can now finish the proof of Theorem \ref{MAIN}. By Lemma \ref{SC} and Theorem \ref{LEVEL} we see that it is enough to show that the Takagi functions satisfy condition $(1)$ in the statement of Lemma \ref{SC}. In fact, condition $(1)$ is satisfied by $T_{a,b}$ whenever $a<1, b\in\{2\}\cup [2,\infty], ab\in (1,2)$. We put the last step of proving Theorem \ref{MAIN} in the following lemma. 
\begin{lma}\label{NONINTEGRAL}
If the parameters $a,b$ satisfy the following conditions
\[
a<1,b\in\{2\}\cup[3,\infty],2>ab>1,
\]
then the Takagi functions
\[
T_{a,b}(x)=\sum_{n=0}^{\infty}a^nT(b^n x)
\]
satisfy the condition 1 in the statement of Lemma \ref{SC}.
\end{lma}
\begin{proof}
For each integer $M>0$ we shall consider the $M$-level intervals
\[
I_M(j)=\left[\frac{j}{2b^M},\frac{j+1}{2b^M}\right],j\in\mathbb{Z}.
\]
If $b\geq 3$, we see that
\[
\frac{1}{2b^M}\geq\frac{3}{2b^{M+1}},
\]
this implies that any interval $I_M(j)$ contains at least two $(M+1)$-level intervals say
\[
I_{M+1}(l), I_{M+1}(l+1).
\]
$l, l+1$ is a pair of integers, one of them is odd and the other is even. It is easy to see that this is also true for $b=2$. Then we see that given any sequance $\omega_i\in\{\pm 1\},i=0,1,2\dots$ it is possible to choose a sequence of integers $j_i,i=0,1,\dots$ such that
\[
I_i(j_i)\subset I_{i-1}(j_{i-1})\subset[0,1],
\]
and that $j_i$ is even if and only if $\omega_i=1$.
 We see that
\[
T'(b^M x)=
\begin{cases}
1 & x\in I_M(j), j \text{ is even},\\
-1 & x\in I_M(j), j \text{ is odd}.
\end{cases}
\]
So for any integer $M>0$ we can find intervals $I_M(j)\subset [0,1]$ such that
\[
D_M(x)=\sum_{n=0}^{M}a^nT(b^n x))'=\sum_{n=0}^{M}a^nb^nT'(b^n x)
\]
is constant on $I_M(j)$ and the value can take all numbers in the following set
\[
U=
\left\{t\in\mathbb{R}:\exists \epsilon_n\in\{\pm1\},n\in\{0,1,2\dots,M\},t=\sum_{n=0}^{M}\epsilon_n(ab)^n\right\}.
\]
Then by lemma \ref{POWERBOUND} we see that the result follows.
\end{proof}

\providecommand{\bysame}{\leavevmode\hbox to3em{\hrulefill}\thinspace}
\providecommand{\MR}{\relax\ifhmode\unskip\space\fi MR }
\providecommand{\MRhref}[2]{%
	\href{http://www.ams.org/mathscinet-getitem?mr=#1}{#2}
}
\providecommand{\href}[2]{#2}

\end{document}